\documentclass[12pt]{amsart}
\usepackage{amsmath,latexsym,amscd,amsbsy,amssymb,amsfonts,amsthm,fleqn,leqno,
euscript, graphicx, texdraw}

\numberwithin{equation}{section}

\newtheorem{thm}{Theorem}[section]
\newtheorem{pro}[thm]{Proposition}

\newtheorem{cor}[thm]{Corollary}
\newtheorem{problem}{Problem}[section]

\theoremstyle{definition}
\newtheorem{dfn}[thm]{Definition}

\theoremstyle{remark}

\begin{document}

\title[Alexandroff manifolds and homogeneous continua]
{Alexandroff manifolds and homogeneous continua}

\author{A. Karassev}
\address{Department of Computer Science and Mathematics,
Nipissing University, 100 College Drive, P.O. Box 5002, North Bay,
ON, P1B 8L7, Canada} \email{alexandk@nipissingu.ca}

\author{V. Todorov}
\address{Department of Mathematics, UACG, $1$ H. Smirnenski blvd.,
$1046$ Sofia, Bulgaria} \email{vtt-fte@uacg.bg}

\author{V. Valov}
\address{Department of Computer Science and Mathematics,
Nipissing University, 100 College Drive, P.O. Box 5002, North Bay,
ON, P1B 8L7, Canada} \email{veskov@nipissingu.ca}

\date{\today}
\thanks{The first author was partially supported by NSERC
Grant 257231-09.}
\thanks{The third author was partially supported by NSERC
Grant 261914-08.}


\begin{abstract}
We prove the following result announced in \cite{tv}: Any homogeneous,
metric $ANR$-continuum is a $V^n_G$-continuum provided $\dim_GX=n\geq 1$ and $\check{H}^n(X;G)\neq 0$, where $G$ is a principal ideal domain.
This implies that any homogeneous $n$-dimensional metric $ANR$-continuum with $\check{H}^n(X;G)\neq 0$ is a $V^n$-continuum in the sense of Alexandroff \cite{ps}.
We also prove that any finite-dimensional homogeneous metric continuum $X$, satisfying $\check{H}^n(X;G)\neq 0$ for some group $G$ and $n\geq 1$,
cannot be separated by
a compactum $K$ with $\check{H}^{n-1}(K;G)=0$ and $\dim_G K\leq n-1$. This provides a partial answer to a question of Kallipoliti-Papasoglu \cite{kp}
whether any two-dimensional homogeneous Peano continuum cannot be separated by arcs.
\end{abstract}
\maketitle\markboth{}{Alexandroff manifolds and homogeneous continua}





\section{Introduction}

Cantor manifolds and stronger versions of Cantor manifolds were introduced to describe some properties of Euclidean manifolds.
According to Bing-Borsuk conjecture \cite{bb} that any homogeneous metric $ANR$-compactum of dimension $n$ is an $n$-manifolds,
finite-dimensional homogeneous metric $ANR$-continua are supposed to share some properties with Euclidean manifolds. One of the first
results in that direction established by Krupski \cite{kru} is that any homogeneous metric continuum of dimension $n$ is a Cantor
$n$-manifold. Recall that
a space $X$ is a
{\em Cantor $n$-manifold} if any partition of $X$ is of dimension $\geq n-1$ \cite{u} (a partition of $X$
is a closed set $P\subset X$ such that $X\backslash P$ is the union of two open disjoint sets).
In other words, $X$ cannot be the union
of two proper closed sets whose intersection is of covering
dimension $\leq n-2$. Stronger versions of Cantor manifolds were considered  by Had\v{z}iivanov \cite{h} and
Had\v{z}iivanov and Todorov \cite{ht}.
But the strongest specification of Cantor manifolds is the notion of $V^n$-continua introduced by
Alexandroff \cite{ps}: a compactum $X$ is a {\em $V^n$-continuum}
if for every two closed disjoint massive subsets $X_0$, $X_1$ of $X$
there exists an open cover
$\omega$ of $X$ such that there is no partition $P$ in $X$ between
$X_0$ and $X_1$ admitting an $\omega$-map into a space $Y$ with
$\dim Y\leq n-2$ ($f\colon P\to Y$ is said to be an $\omega$-map if there exists
an open cover $\gamma$ of $Y$ such that $f^{-1}(\gamma)$ refines $\omega$). Recall that a
massive subset of $X$ is a set with non-empty interior in $X$.

More
general concepts of the above notions were considered in \cite{kktv}. In particular, we are
going to use the following one, where
$\mathcal{C}$ is a class of topological spaces.

\begin{dfn}\label{dfn1}
A space $X$ is an {\em Alexandroff manifold with respect to $\mathcal C$}
(br., {\em Alexandroff $\mathcal C$-manifold}) if for every two closed, disjoint, massive subsets $X_0$, $X_1$ of $X$
there exists an open cover
$\omega$ of $X$ such that there is no partition $P$ in $X$ between
$X_0$ and $X_1$ admitting an $\omega$-map onto a space $Y\in\mathcal C$.
\end{dfn}

In this paper we continue investigating to what extend homogeneous continua have common properties with Euclidean manifolds.
One of the main questions in this direction is whether any homogeneous $n$-dimensional metric $ANR$-compactum $X$ is
a $V^n$-continuum, see \cite{tv}. A partial answer of this question, when \v{C}ech cohomology group $\check{H}^n(X)$ is
non-trivial, was announced in \cite{tv}. One of the aims of the paper is to provide the proof of this fact, see Section 3. Our proof is based on
the properties of $(n,G)$-bubbles and $V^n_G$-continua investigated in Section 2. We also provide a partial answer to a question of
Kallipoliti-Papasoglu \cite{kp}.

\section{$(n,G)$-bubbles and $V^n_G$-continua}

In this section we investigate the connection between $(n,G)$-bubbles and $V^n_G$-continua.

For every abelian group $G$ let $\dim_GX$ be the cohomological dimension of $X$ with respect to $G$, and $\check{H}^n(X;G)$ denotes the
reduced $n$-th \v{C}ech cohomology group of $X$ with coefficients in $G$.

Reformulating the original definition of Kuperberg \cite{kup},
Yokoi \cite{yo} provided the following definition (see also \cite{ch} and \cite{kr}):
\begin{dfn}
If $G$ is an abelian group and $n\geq 0$, a compactum $X$
is called an {\em $(n,G)$-bubble} if $\check{H}^n(X;G)\neq 0$ and $\check{H}^n(A;G)=0$ for every proper closed subset $A$ of $X$.
Following \cite{vt1} we say that a compactum $X$ is a {\em generalized $(n,G)$-bubble} provided there exists a surjective map $f\colon X\to Y$ such that
the homomorphism $f^*\colon\check{H}^n(Y;G)\to\check{H}^n(X;G)$ is nontrivial, but $f_A^*(\check{H}^n(Y;G))=0$ for any proper closed subset $A$ of $X$,
where $f_A$ is the restriction of $f$ over $A$.
\end{dfn}
We also need the following notion:
\begin{dfn}
A compactum $X$ is said to be a {\em $V^n_G$-continuum} \cite{ss}
if for every two closed, disjoint, massive subsets $X_0$, $X_1$ of $X$ there exists an open cover $\omega$ of $X$ such that any
partition $P$ in $X$ between $X_0$ and $X_1$ does not admit an $\omega$-map $g$ onto a space $Y$ with
$g^*\colon\check{H}^{n-1}(Y;G)\to\check{H}^{n-1}(P;G)$ being trivial.
\end{dfn}
Since $\check{H}^{n-1}(Y;G)=0$ for any compactum $Y$ with
$\dim_GY\leq n-2$, $V^n_G$-continua are Alexandroff manifolds with respect to the class  $D^{n-2}_G$ of
all spaces of dimension $\dim_G\leq n-2$. Moreover, if $X\in V^n_G$, then for every partition $C$ of $X$ we have $\check{H}^{n-1}(C;G)\neq 0$.
The last observation implies $\dim_GX\geq n$ provided $X$ is a metric $V^n_G$-compactum such that  either $X\in ANR$ or
$\dim X<\infty$ and $G$ is countable.
Indeed, if $\dim_GX\leq n-1$, then each $x\in X$ has a local base of open sets $U$ whose boundaries are of dimension $\dim_G\leq n-2$, see \cite{dk}.
Hence, any such a boundary $\Gamma$ is a partition of $X$ with $\check{H}^{n-1}(\Gamma;G)=0$.

Next theorem was established in \cite[Theorem 3]{ss} for finite-dimensional metric $(n,G)$-bubbles. Let us note that, according to \cite{yo},
the examples of
Dranishnikov \cite{dr} and Dydak-Walsh \cite{dw} show the existence of an infinite-dimensional $(n,\mathbb Z)$-bubble with $n\geq 5$.

\begin{thm}\label{2.1}
Any generalized $(n,G)$-bubble $X$ is a $V^n_G$-continuum.
\end{thm}

\begin{proof}
Let $f\colon X\to Y$ be a map such that $f^*(\check{H}^n(Y;G))\neq 0$ and $f_A^*(\check{H}^n(Y;G))=0$ for any proper closed set $A\subset X$.
If $\omega$ is a finite open cover of  a closed
set $Z\subset X$, we denote by $|\omega|$ and $p_\omega$, respectively, the nerve of $\omega$ and a map from $Z$ onto $|\omega|$ generated by a
partition of unity subordinated to $\omega$.
Furthermore, if $C\subset Z$ and $\omega(C)=\{W\in\omega:W\cap C\neq\varnothing\}$, then $p_{\omega(C)}\colon C\to |\omega(C)|$ is the restriction
$p_\omega|C$.
Recall also that $p_\omega$ generates maps $p_\omega^*\colon\check{H}^k(|\omega|;G)\to\check{H}^k(Z;G)$, $k\geq 0$. Moreover, if
$q_\omega\colon Z\to|\omega|$ is a map generating by (another) partition of unity subordinated to $|\omega|$, then $p_\omega$ and $q_\omega$
are homotopic. So, $p_\omega^*=q_\omega^*$.

\textit{Claim $1$. For every non-empty open sets $U_1$ and $U_2$ in $X$ with $\overline{U}_1\cap\overline{U}_2=\varnothing$ there exist an
open cover $\omega$ of $X\backslash(U_1\cup U_2)$, a map $p_\omega\colon X\backslash(U_1\cup U_2)\to|\omega|$  and an element
$e\in\check{H}^{n-1}(|\omega|;G)$ such that $p_{\omega(C)}^*(i_C^*(\mathrm{e}))\neq 0$
for every partition $C$ of $X$ between $\overline{U}_1$ and $\overline{U}_2$, where $i_C$ is the inclusion $|\omega(C)|\hookrightarrow|\omega|$}.

To prove this claim we follow the arguments from the proof of \cite[Theorem]{vt1}.
Let $U_1$ and $U_2$ be non-empty open subsets of $X$ with disjoint closures, and $i_k:F_k\hookrightarrow X$ be the inclusion of
$F_k=X\backslash U_k$ into $X$, $k=1,2$. Consider the Mayer-Vietoris exact sequence
{ $$
\begin{CD}
\check{H}^{n-1}(F_1\cap F_2;G)@>{{\delta}}>>\check{H}^{n}(X;G)@>{{j}}>>\check{H}^{n}(F_1;G)\oplus\check{H}^{n}(F_2;G)\rightarrow
\end{CD}
$$}\\
with $j=(i_1^*,i_2^*)$, and
choose a non-zero element $\mathrm{e}_1\in f^*(\check{H}^n(Y;G))\subset\check{H}^n(X;G)$. For each $k=1,2$ we have the commutative diagram,
where $\delta_k$ is the inclusion of $f(F_k)$ into $Y$:
{ $$
\begin{CD}
\check{H}^{n}(Y;G)@>{{f^*}}>>\check{H}^{n}(X;G)\\
@ VV{\delta_k^*}V
@VV{i_k^*}V\\
\check{H}^{n}(f(F_k);G)@>{{f_{F_k}^*}}>>\check{H}^{n}(F_k;G).
\end{CD}
$$}\\
So  $i_k^*(\mathrm{e}_1)=0$, $k=1,2$, which yields
$\mathrm{e}_1=\delta(\mathrm{e}_2)$ for some non-zero element
$\mathrm{e}_2\in\check{H}^{n-1}(F_1\cap F_2;G)$. Then there exist an open cover $\omega$ of $F_1\cap F_2=X\backslash(U_1\cup U_2)$, a map
$p_\omega\colon F_1\cap F_2\to|\omega|$ and
$\mathrm{e}\in\check{H}^{n-1}(|\omega|;G)$ with $p_\omega^*(\mathrm{e})=\mathrm{e}_2$.

Let $C$ be a partition of $X$ between $\overline{U}_1$ and $\overline{U}_2$. So, $X=P_1\cup P_2$ and $C=P_1\cap P_2$, where each $P_k$ is
a closed subset of $X$ containing $\overline{U}_k$, $k=1,2$. Denote by $i:C\hookrightarrow F_1\cap F_2$, $i_1:P_1\hookrightarrow F_2$
and $i_2:P_2\hookrightarrow F_1$ the corresponding inclusions.
Then we have the following commutative diagram, whose rows are Mayer-Vietoris
sequences:

{ $$
\begin{CD}
\check{H}^{n-1}(F_1\cap F_2;G)@>{{\delta}}>>\check{H}^{n}(X;G)@>{{j}}>>\check{H}^{n}(F_2;G)\oplus\check{H}^{n}(F_1;G)\\
@ VV{i^*}V
@VV{id}V@VV{i_1^*\oplus i_2^*}V\\
\check{H}^{n-1}(C;G)@>{{\delta_1}}>>\check{H}^{n}(X;G)@>{{j_1}}>>\check{H}^{n}(P_1;G)
\oplus\check{H}^{n}(P_2;G).
\end{CD}
$$}\\
Obviously, $$\delta_1(i^*(\mathrm{e}_2))=id(\delta(\mathrm{e}_2))=\mathrm{e}_1\neq 0.\leqno{(1)}$$ On the other hand, the commutativity of the diagram

{ $$
\begin{CD}
\check{H}^{n-1}(|\omega|;G)@>{{p_{\omega}^*}}>>\check{H}^{n-1}(F_1\cap F_2;G)\\
@ VV{i_C^*}V
@VV{i^*}V\\
\check{H}^{n-1}(|\omega(C)|;G)@>{{p_{\omega(C)}^*}}>>\check{H}^{n-1}(C;G)
\end{CD}
$$}\\
implies that $p_{\omega(C)}^*(i_C^*(\mathrm{e}))=i^*(p_{\omega}^*(\mathrm{e}))=i^*(\mathrm{e}_2)$. Therefore, according to (1),
$p_{\omega(C)}^*(i_C^*(\mathrm{e}))\neq 0$. This completes the proof of Claim 1.

Now, we can show that $X\in V^n_G$. Let $U_1$ and $U_2$ be non-empty open subsets of $X$ with disjoint closures. Then there exists
a finite open cower $\omega$ of $X\backslash(U_1\cup U_2)$, a map $p_\omega\colon X\backslash(U_1\cap U_2)\to|\omega|$ and an element
$\mathrm{e}\in\check{H}^{n-1}(|\omega|;G)$ satisfying the conditions from Claim 1.
For each $W\in\omega$ let $h(W)$ be an open subset of $X$ extending $W$. So, $\gamma=\{h(W):W\in\omega\}\cup\{U_1,U_2\}$ is a
finite open cover of $X$ whose restriction on $X\backslash(U_1\cup U_2)$ is $\omega$.

Suppose there exists a partition $C$ of $X$ between $\overline{U}_1$ and
$\overline{U}_2$ admitting a $\gamma$-map $g$ onto a space $T$ with $g^*(\check{H}^{n-1}(T;G))=0$. Thus, we can find a finite open cover $\alpha$
of $T$  such that
$\beta=g^{-1}(\alpha)$ is refining $\omega$. Let $p_\beta\colon C\to|\beta|$ be a map onto the nerve of $\beta$ generated by a partition
of unity subordinated to $\beta$. Obviously, the function $V\in\alpha\rightarrow g^{-1}(V)\in\beta$ generates a
a simplicial homeomorphism $g^{\alpha}_\beta\colon|\alpha|\to|\beta|$. Then the maps $p_\beta$ and
$g_\alpha=g^{\alpha}_\beta\circ\pi_\alpha\circ g$, where $\pi_\alpha\colon T\to|\alpha|$ is a map generated by a partition of unity
subordinated to $|\alpha|$,
are homotopic. Hence, $p_{\beta}^*=g^*\circ\pi_{\alpha}^*\circ (g^{\alpha}_\beta)^*$.
Because $g^*\colon\check{H}^{n-1}(T;G)\to\check{H}^{n-1}(C;G)$
is a trivial map, the last equality implies that so is the map $p_{\beta}^*\colon\check{H}^{n-1}(|\beta|;G)\to\check{H}^{n-1}(C;G)$.
On the other hand, since $\beta$
refines $\omega$, we can find a map $\varphi_\beta\colon|\beta|\to |\omega(C)|$ such that $p _{\omega(C)}$ and $\varphi_\beta\circ p_{\beta}$
are homotopic. Therefore, $p_{\omega(C)}^*=p_{\beta}^*\circ\varphi_\beta^*$.
According to Claim 1,
$p_{\omega(C)}^*(\mathrm{e}_C)\neq 0$, where $\mathrm{e}_C$ is the element $i_C^*(\mathrm{e})\in\check H^{n-1}(|\omega(C)|;G)$. So,
$p_{\beta}^*(\varphi_\beta^*(\mathrm{e}_C))\neq 0$, which contradicts the triviality of $p_{\beta}^*$.
\end{proof}

We can extend the definition of $V^n_G$-continua as follows:
\begin{dfn}
A compactum $X$ is said to be a
{\em $V^n_G$-continuum with respect to a given class $\mathcal A$} if
for every two closed, disjoint, massive subsets $X_0$, $X_1$ of $X$ there exists an open cover $\omega$ of $X$ such that any
partition $P$ in $X$ between $X_0$ and $X_1$ does not admit an $\omega$-map $g$ onto a space $Y\in\mathcal A$ with
$g^*\colon\check{H}^{n-1}(Y;G)\to\check{H}^{n-1}(P;G)$ being trivial.
\end{dfn}
Recall that a metric space  $X$ is {\em strongly $n$-universal} if any
map $g:K\to X$, where $K$ is a metric compactum of dimension $\dim K\leq n$, can be approximated by embeddings.

\begin{thm}\label{2.2}
Let $X$ be a metric compactum containing a strongly $n$-universal dense subspace $M$
such that $M$ is an absolute extensor for $n$-dimensional compacta with $n\geq 1$. Then $X$ is a $V^n_G$-continuum with
respect to the class $D^{n-1}_G$ of all spaces of dimension $\dim_G\leq n-1$. In particular, $X$ is an Alexandroff manifold
with respect to the class $D^{n-2}_G$.
\end{thm}

\begin{proof}
Suppose that $X$ is not a $V^n_G$-continuum with respect to the class $D^{n-1}_G$. So, we can find open sets
$U$ and $V$ in $X$ with disjoint closures such that
for every $\epsilon>0$ there exists a partition $C_\epsilon$ between $\overline{U}$ and $\overline{V}$ admitting an $\epsilon$-map $g_\epsilon$
onto a space $Y_\epsilon\in D^{n-1}_G$ such that $g_\epsilon^*\colon\check{H}^{n-1}(Y_\epsilon;G)\to\check{H}^{n-1}(C_\epsilon;G)$
is trivial.

Consider two different points $a, b$ from the $n$-sphere $\mathbb S^n$, and a map $f\colon\mathbb S^n\to M$ with
$f(a)\in U\cap M$ and $f(b)\in V\cap M$ (such a map exists because $M$ is an absolute extensor for $n$-dimensional compacta).
Since $M$ is strongly $n$-universal, we can approximate $f$ by a homeomorphism
$h\colon\mathbb S^n\to M$ such that $h(a)\in U$ and $h(b)\in V$. Therefore, $K_\epsilon=C_\epsilon\cap h(\mathbb S^n)$
is a partition of $h(\mathbb S^n)$ between $h(\mathbb S^n)\cap\overline{U}$ and $h(\mathbb S^n)\cap\overline{V}$.
Then $Z=g_\epsilon(K_\epsilon)$ is a closed
subset of $Y_\epsilon$, and since $\dim_GY_\epsilon\leq n-1$, $i_Z^*\colon\check{H}^{n-1}(Y_\epsilon;G)\to\check{H}^{n-1}(Z;G)$
is a surjective map, where
$i_Z:Z\hookrightarrow Y_\epsilon$ is the inclusion. So, we have the following commutative diagram with $g_{K_\epsilon}=g|K_\epsilon$ and
$i_{K_\epsilon}:K_\epsilon\hookrightarrow C_\epsilon$:
{ $$
\begin{CD}
\check{H}^{n-1}(Y_\epsilon;G)@>{{g_\epsilon^*}}>>\check{H}^{n-1}(C_\epsilon;G)\\
@ VV{i_Z^*}V
@VV{i_{K_\epsilon}^*}V\\
\check{H}^{n-1}(Z;G)@>{{g_{K_\epsilon}^*}}>>\check{H}^{n-1}(K_\epsilon;G).
\end{CD}
$$}\\
Because $g_\epsilon^*$ is trivial and $i_Z^*$ is surjective, $g_{K_\epsilon}^*$ is also trivial. Hence, for every $\epsilon>0$
there exists a partition $K_\epsilon$
between $h(\mathbb S^n)\cap\overline{U}$ and $h(\mathbb S^n)\cap\overline{V}$ admitting an $\epsilon$-map $g_{K_\epsilon}$ onto a space $Z$ such
that $g_{K_\epsilon}^*\colon\check{H}^{n-1}(Z;G)\to\check{H}^{n-1}(K_\epsilon;G)$ is trivial. This means that $\mathbb S^n$ is not
a $V^n_G$-continuum. On the other hand, $\mathbb S^n$ is an $(n,G)$-bubble for all $G$. So, by Theorem \ref{2.1} ,
$\mathbb S^n$ is a $V^n_G$-continuum -  a contradiction.
\end{proof}

\begin{cor}
Let $X$ be either the universal Menger compactum $\mu^n$ or $X$ be a metric compactification of the universal N\"{o}beling space $\nu^n$.
Then  $X$ is a $V^n_G$-continuum with
respect to the class $D^{n-1}_G$ for any $G$. Moreover, $\mu^n$ is not a $V^n_G$-continuum.
\end{cor}

\begin{proof}
Since both $\mu^n$ and $\nu^n$ are strongly $n$-universal absolute extensors for $n$-dimensional compacta, it follows from Theorem \ref{2.2}
that $X$ is a $V^n_G$-continuum with respect to the class $D^{n-1}_G$. To show that $\mu^n$ is not a $V^n_G$-continuum,
it suffices to find a partition $E$ of $\mu^n$ with trivial
$\check{H}^{n-1}(E;G)$. One can show the existence of such partitions using the geometric construction of the Menger compactum. We provide
a proof of this fact using Dranishnikov's results from \cite{dr1}. Indeed, by \cite[Theorem 2]{dr1}, there exists a map $g\colon\mu^n\to\mathbb I^\infty$  such that $g^{-1}(P)$ is homeomorphic to $\mu^n$ for
any $AR$-space $P\subset\mathbb I^\infty$. If $P\in AR$ is a partition of $\mathbb I^\infty$, then $g^{-1}(P)$
is a partition of $\mu^n$ homeomorphic to $\mu^n$. Hence, $\check{H}^{n-1}(g^{-1}(P);G)=0$.
\end{proof}

\section{Homogeneous continua}

In this section we prove that some homogeneous continua are $V^n_G$-continua. Recall that
a space $X$ is said to be {\em homogeneous} if for every two points $x,y\in X$ there exist a homeomorphism $h:X\to X$ with $h(x)=y$.
Krupski \cite{kr1} conjectured that any $n$-dimensional, homogeneous metric $ANR$-continuum is a $V^n$-continuum.
Next result provides a partial solution to Krupski's conjecture and a partial answer to Question 2.4 from \cite{tv}.

\begin{thm}
Let $X$ be a homogeneous, metric $ANR$-continuum with $\dim_GX=n\geq 1$ such that $\check{H}^n(X;G)\neq 0$, where $G$ is a principal ideal domain.
Then $X$ is a $V^n_G$-continuum.
\end{thm}

\begin{proof}
According to \cite[Theorem 3.3]{yo}, any space $X$ satisfying the conditions from this theorem is an $(n,G)$-bubble. Hence, by Theorem \ref{2.1},
$X\in V^n_G$.
\end{proof}

Bing and Borsuk \cite{bb} raised the question whether no compact acyclic in dimension $n-1$ subset of $X$ separates $X$, where
$X$ is a metric  $n$-dimensional homogeneous $ANR$-continuum. Yokoi \cite[Corollary 3.4]{yo}
provided a partial positive answer to this question in
the case $X$ is
a homogeneous metric $n$-dimensional $ANR$-continuum such that $\check{H}^{n}(X;\mathbb Z)\neq 0$. Next proposition is a version of Yokoi's
result when $X$ is not necessarily $ANR$.

\begin{pro}\label{sep}
Let $X$ be a finite-dimensional homogeneous metric continuum
with $\check{H}^{n}(X;G)\neq 0$. Then $\check{H}^{n-1}(C;G)\neq 0$ for
any partition $C$ of $X$ such that $\dim_GC\leq n-1$.
\end{pro}

\begin{proof}
Suppose there exists a partition $C$ of $X$ such that $\check{H}^{n-1}(C;G)=0$ and $\dim_GC\leq n-1$. The last inequality implies that
the inclusion homomorphism $\check{H}^{n-1}(C;G)\rightarrow\check{H}^{n-1}(A;G)$ is an epimorphism for every closed set $A\subset C$.
So, $\check{H}^{n-1}(A;G)=0$ for all closed subsets of $C$. Therefore, we may assume that $C$ does not have any interior points.
Since $\check{H}^{n}(X;G)\neq 0$, according to \cite[Theorem 2]{ss}, there exists a compact subset $K\subset X$ with $K\in V^n_G$.
Since $X$ is homogeneous, we may also assume that $K\cap C\neq\varnothing$. Observe that $z\in K\backslash C$ for some $z$.
Indeed, the inclusion $K\subset C$ would imply that $\check{H}^{n-1}(P;G)=0$ for every
partition $P$ of $K$. Let $X\backslash C=U\cup V$ and $z\in V$,
where $U$ and $V$ are nonempty, open and disjoint sets in $X$. Then the Effros theorem \cite{ef} allows us to push $K$ towards $U$ by a
small homeomorphism $h\colon X\to X$ so that the image $h(K)$ meets both $U$ and $V$ (see the proof of Lemma 2 from \cite{kru}
for a similar application of Effros' theorem). Therefore, $S=h(K)\cap C$ is a partition of $h(K)$ and $\check{H}^{n-1}(S;G)=0$
because $S\subset C$, a contradiction.
\end{proof}

Proposition~\ref{sep} provides a partial answer to a question of Kallipoliti-Papasoglu \cite{kp} whether homogeneous two-dimensional
metric locally connected continua cannot be separated by arcs.

\begin{cor}
No finite-dimensional metric homogeneous continuum $X$ with $\check{H}^{2}(X;G)\neq 0$ can be separated by any one-dimensional compactum
$C$ with $\check{H}^{1}(C;G)=0$.
\end{cor}

\section{Some remarks and problems}

The class of $(n,G)$-bubbles is stable in the sense of the following proposition.

\begin{pro}
Let $X$ be a metric compactum  admitting an $\epsilon$-map onto an $(n,G)$-bubble for any $\epsilon>0$. Then $X$ is also an
$(n,G)$-bubble.
\end{pro}

\begin{proof}
First, let us show that $\check{H}^{n}(X;G)\neq 0$. Take any open cover $\omega$ of $X$ and let $\epsilon$ be the Lebesgue number of $\omega$.
There exists a surjective $\epsilon$-map $f\colon X\to Y_\epsilon$ with $Y_\epsilon$ being an $(n,G)$-bubble. Since
$\check{H}^{n}(Y_\epsilon;G)\neq 0$, we can find an open cover $\alpha$ of $Y_\epsilon$ such that $\check{H}^{n}(|\alpha|;G)\neq 0$ (we use the
notations from the proof of Theorem \ref{2.1}). Then $\beta=f^{-1}(\alpha)$ is an open cover of $X$ refining $\omega$ such that $|\beta|$ is homeomorphic
to $|\alpha|$. So, $\check{H}^{n}(|\beta|;G)\neq 0$, which implies $\check{H}^{n}(X;G)\neq 0$.

Suppose now that $A$ is a proper closed subset of $X$ and $\gamma$ an open (in $A$) cover of $A$. Extend each $U\in\gamma$ to an open set $V(U)$ in
$X$ and let $W=\cup\{V(U):U\in\gamma\}$. We can suppose that $W\neq X$. Choose a surjective $\eta$-map $g\colon X\to Y_\eta$ such that $Y_\eta$ is an
$(n,G)$-bubble with $\eta$ being a positive number smaller than both  ${\rm dist} (A,X\backslash W)$ and the Lebesgue number of $\gamma$. Then $B=g(A)$
is a proper closed subset of $Y_\eta$ such that $g^{-1}(B)\subset W$. There exists an open cover $\theta$ of $B$ such that
the family $\delta=\{g^{-1}(G)\cap A:G\in\theta\}$ is an open cover of $A$ refining $\gamma$. Obviously, $|\delta|$ is homeomorphic to $|\theta|$.
Since $\check{H}^{n}(B;G)=0$, we have $\check{H}^{n}(|\theta|;G)=\check{H}^{n}(|\delta|;G)=0$. Hence $\check{H}^{n}(A;G)=0$, which completes the proof.
\end{proof}

Now, we are going to discuss some problems. The main question suggested by the results from this paper is to remove in Theorem 3.1
some of the conditions about $X$. Since, according to Corollary 2.3, $\mu^n$ is not a $V^n_G$-continuum for any $G$, the condition
$X$ to be an $ANR$ cannot be removed. So, we have the following question.

\begin{problem}
Let $X$ be a homogeneous metric $ANR$-continuum $X$ with $\dim_GX=n$, where $G$ is any abelian group. Is  $X$ a $V^n_G$-continuum?
\end{problem}

Since any $V^n_G$-continuum with respect to the class $D^{n-1}_G$ is $V^n$, next question is still interesting.

\begin{problem}
Let $X$ be a homogeneous metric continuum $X$ with $\dim_GX=n$. Is  $X$ a $V^n_G$-continuum with
respect to the class $D^{n-1}_G$? What  if $\check{H}^n(X;G)\neq 0$?
\end{problem}

Another question is whether finite-dimensionality can be removed from the result of Stefanov \cite{ss} which
was applied above.

\begin{problem}
Let $X$ be a metrizable compactum with $\check{H}^{n}(X;G)\neq 0$ for some group $G$ and $n\geq 1$. Does $X$ contain a $V^n_G$-continuum?
\end{problem}

We can show that any finite simplicial complex is a generalized $(n,G)$-bubble if and only if it is an $(n,G)$-bubble. So, our last question is whether
this remain true for all metric compacta.

\begin{problem}
Is there any metric compactum $X$ which is a generalized $(n,G)$-bubble but not an $(n,G)$-bubble?
\end{problem}



\begin{thebibliography}{999}


\bibitem{ps} P.~S.~Alexandroff, \textit{Die Kontinua $(V^p)$ - eine
Versch\"{a}rfung der Cantorschen Mannigfaltigkeiten}, Monatshefte
fur Math. \textbf{61} (1957), 67--76 (German).


\bibitem{bb}
R.~H.~Bing and K.~Borsuk, \textit{Some remarks concerning topological homogeneous spaces}, Ann. of Math. \textbf{81} (1965), no. 1, 100--111.

\bibitem{ch}
J.~Choi, \textit{Properties of $n$-bubles in $n$-dimensional compacta and the existence of $(n-1)$-bubles in
$n$-dimensional $clc^n$ compacta}, Top. Proceed. \textbf{23} (1998), 101-120.

\bibitem{dr}
A.~Dranishnikov, \textit{On a problem of P.S. Aleksandrov}, Mat. Sb. \textbf{135} (1988), no. 4, 551--557 (in Russian).

\bibitem{dr1}
A.~Dranishnikov, \textit{Universal Menger compacta and universal mappings}, Mat. Sb. \textbf{129} (1987), no. 1, 121--139 (in Russian).

\bibitem{dw}
J.~Dydak and J.~Wash, \textit{Infinite-dimensional compacta having cohomological dimension two: an application of the Sullivan conjecture},
Topology \textbf{32} (1993), 93--104.

\bibitem{dk}
J.~Dydak and A.~Koyama, \textit{Cohomological dimension of locally connected compacta}, Topol. Appl. \textbf{113} (2001), 39--50.

\bibitem{ef} E. G. Effros, \textit{Transformation groups and $C^*$-algebras},
Ann. of Math. \textbf{81} (1965), 38--55.
















\bibitem{h} N.~Had\v{z}iivanov, \textit{Strong Cantor manifolds},
C. R. Acad. Bulgare Sci. \textbf{30} (1977), 1247--1249 (Russian).


\bibitem{ht} N.~Had\v{z}iivanov and V. Todorov, \textit{On
non-Euclidean manifolds}, C. R. Acad. Bulgare Sci. \textbf{33}
(1980), 449--452 (Russian).










\bibitem{kp}
M.~Kallipoliti and P.~Papasoglu, \textit{Simply connected homogeneous continua are not separated by arcs},
Top. Appl. \textbf{154} (2007), 3039--3047.

\bibitem{kr}
U.~Karimov and D.~Repov\v{s}, \textit{On $\check{H}^n$-bubles in $n$-dimensional compacta},
Colloq. Math. \textbf{75} (1998), 39--51.

\bibitem{kktv}
A.~Karassev, P.~Krupski, V.~Todorov and V.~Valov, \textit{Generalized Cantor manifolds and homogeneity},
Houston J. Math. \textbf{38} (2012), no. 2, 583--609.


\bibitem{kr1}
P.~Krupski, Private communication, 2007.

\bibitem{kru} P.~Krupski, \textit{Homogeneity and Cantor manifolds},
Proc. Amer. Math. Soc. \textbf{109} (1990), 1135--1142.

\bibitem{kup}
W.~Kuperberg, \textit{On certain homological properties of finite-dimensional compacta.
Carries, minimal carries and bubles}, Fund. Math. \textbf{83} (1973), 7--23.















\bibitem{ss}
S.~Stefanov, \textit{A cohomological analogue of $V^n$-continua and a theorem of Mazurkiewicz},
Serdica \textbf{12} (1986), no. 1, 88--94 (in Russian).

\bibitem{tv}
V.~Todorov and V.~Valov, \textit{Generalized Cantor manifolds and indecomposable continua},
Questions and Answers in Gen. Topolol., accepted.


\bibitem{vt1}
V.~Todorov, \textit{Irreducibly cyclic compacta and Cantor manifolds}, Proc. of the Tenth Spring Conf.
of the Union of Bulg. Math., Sunny Beach, April 6-9 (1981).



\bibitem{u} P.~Urysohn, \textit{Memoire sur les multiplicites
cantoriennes}, Fund. Math. \textbf{7} (1925), 30--137.

\bibitem{yo}
K.~Yokoi, \textit{Bubbly continua and homogeneity}, Houston J. Math. \textbf{29} (2003), no. 2, 337--343.

\end{thebibliography}
\end{document}